%% file: main.tex
\documentclass[preprint,12pt]{elsarticle}
\normalfont
\usepackage{amsmath,amssymb,amsthm,booktabs,hyperref}
\usepackage{graphicx}
\usepackage[margin=1in]{geometry}
\hypersetup{pdfauthor={Chandrasekhar Gokavarapu, Lakshmi Gayathri Isukapalli, Janardhana Rao Kolli, Fathima Sulthana Md, Sai Durga Rolla}}
\newtheorem{theorem}{Theorem}
\newtheorem{lemma}{Lemma}
\newtheorem{proposition}[theorem]{Proposition}
\newcommand{\lm}{\operatorname{lm}}
\newcommand{\supp}{\operatorname{supp}}
\newcommand{\spanf}{\operatorname{span}}
\journal{Theoretical Computer Science}
\begin{document}
\begin{frontmatter}
\title{Succinct Strict Pair Certificates for Reduced Gr\"obner Bases}
\author[math]{Chandrasekhar Gokavarapu\corref{cor}\fnref{lect}}
\cortext[cor]{Corresponding author. \href{https://gcrjy.ac.in/facultyprofile.php?facultyid=204\&departmentid=44}{Institutional profile}; \href{https://scholar.google.com/citations?user=8EQ6oz8AAAAJ\&hl=en}{Google Scholar}; \href{https://www.researchgate.net/profile/Chandrasekhar-Gokavarapu}{ResearchGate}; \href{https://orcid.org/0009-0006-5306-371X}{ORCID 0009-0006-5306-371X}.}
\ead{chandrasekhargokavarapu@gmail.com}
\ead{chandrasekhargokavarapu@gcrjy.ac.in}
\author[math]{Lakshmi Gayathri Isukapalli\fnref{lect,lakshmi}}
\ead{laxmi67@gcrjy.ac.in}
\author[math]{Janardhana Rao Kolli\fnref{lect,janardhana}}
\ead{kjrao2013@gmail.com}
\author[math]{Fathima Sulthana Md\fnref{lect,fathima}}
\ead{fathima@gcrjy.ac.in}
\author[math]{Sai Durga Rolla\fnref{student}}
\ead{saidurgarolla@gmail.com}
\affiliation[math]{organization={Department of Mathematics, Government College (Autonomous)},
city={Rajahmundry}, postcode={533105}, state={Andhra Pradesh}, country={India}}
\fntext[lect]{Lecturer in Mathematics.}
\fntext[lakshmi]{\href{https://orcid.org/0009-0007-3833-0995}{ORCID 0009-0007-3833-0995}.}
\fntext[janardhana]{\href{https://orcid.org/0009-0001-9922-7293}{ORCID 0009-0001-9922-7293}.}
\fntext[fathima]{\href{https://orcid.org/0009-0006-8092-524X}{ORCID 0009-0006-8092-524X}.}
\fntext[student]{M.Sc. Mathematics student, 2025--2027 batch.
\href{https://orcid.org/0009-0000-1849-5265}{ORCID 0009-0000-1849-5265}.}
\begin{abstract}
We classify verification of supplied strict critical-pair witnesses
for Gr\"obner bases in a graph model with shared polynomial and
exponent computations. Over the rationals, unrestricted verification
in two variables is polynomial-time many-one equivalent to identity
testing for powerful skew circuits over the integers. The reverse
reduction uses a fixed two-element reduced Gr\"obner basis and
preserves the formal monomial bound on every witness it constructs.
The forward reduction gives a one-sided-error randomized
polynomial-time verifier without a support bound. We also prove
matching bounds on certificate length for explicit families. For a
bivariate family indexed by \(k\), literal shifts require
\(\Theta(k^2)\) bits, shared exponent computations require
\(\Theta(k)\) bits, and an expanded certificate requires
\(\Omega(k2^k)\) bits. For a multivariate family with
\(K=\sum_i k_i\), the bounds are
\(\Theta(nK+\sum_i k_i^2)\) for full shift vectors and
\(\Theta(K\log_2(n+1)+\sum_i k_i^2)\) for sparse shift labels.
\end{abstract}
\begin{keyword}
Gr\"obner bases \sep Symbolic preprocessing \sep Sparse elimination \sep Finite fields \sep Reduction certificates
\MSC[2020] Primary 13P10 \sep Secondary 68W30, 65F50
\end{keyword}
\end{frontmatter}

\section{Introduction}
A direct verification of a Gr\"obner basis can use strict standard
representations for its critical pairs. Each representation expresses
an \(S\)-polynomial using monomial multiples below the pair least
common multiple. The certificate size depends on whether the verifier
requires an expanded cofactor tuple or accepts a graph that shares
intermediate polynomial identities.

Our main result classifies verification of a supplied witness in
the shared-exponent model (Theorem~\ref{thm:verification-pit}).
A verifier can check the strict monomial bound by propagating
bounds through the graph. Checking the polynomial identity is
equivalent, under polynomial-time many-one reductions, to identity
testing for powerful skew circuits over the integers.

K\"onig and Lohrey introduced powerful skew circuits and gave
randomized parallel identity testing for them, with applications
to compressed strings and group word problems \cite{KLPowerful}.
Balaji et al.\ subsequently studied identity testing through
evaluation at roots of unity and applied their technique to
powerful skew circuits \cite{CIT21}. The forward direction of our
theorem applies powerful skew identity testing to strict pair
witnesses. The reverse direction uses a fixed reduced basis and
adds a circuit output multiplied by one generator to a valid
witness. The choice of lexicographic order makes every formal
summand in the added graph strictly admissible, regardless of
its degree in the lower variable. This construction locates
the complexity in checking an arbitrary supplied graph. Finding
a witness, minimizing its encoding, and deciding whether the
basis is Gr\"obner are different problems.

Hofstadler and Verron minimize the number of terms in expanded
ideal-membership cofactors \cite{HV24}. Arithmetic-circuit proof
systems also use shared computations \cite{GPIPS}. Our size
results concern graphs restricted to addition, scalar
multiplication and monomial shifts, with a bound on the formal
summands. We construct reduced bases with a unique dense strict
cofactor tuple and prove matching bit bounds for explicit shift
formats. For the bivariate family, sharing exponent computations
reduces the minimum length from quadratic to linear in the
family parameter. The multivariate bounds quantify the cost of
storing full and sparse shift labels.

\input{syzygy}
\input{encoding}
\input{sharp_length}
\input{shared_exponents}
\input{identity_testing}

\section{Conclusion and scope}
Theorem~\ref{thm:verification-pit} gives an unrestricted randomized
verification algorithm and equates deterministic polynomial-time
verification with powerful skew identity testing. The reverse
reduction holds for a fixed reduced basis.

For the two reduced-basis families the strict pair bound forces
cofactor support to grow exponentially in the number of additions
needed by a shared witness. The exact univariate count is \(k\).
For the multivariate family the lower and upper bounds are \(K\)
and \(2K-1\), and the construction is valid under a bound
propagated before cancellation. The companion verifier accepts
four instances of each construction and rejects four
multivariate witnesses that reach the pair lcm before cancellation.
These checks illustrate the stated representation model and do
not extend its lower bounds to other certificate formats.

The bit-length separation holds for the stated sparse encoding.
The verifier bound uses the small intermediate supports of these
particular graphs. The appendix timings characterize the finite-field
prototype. A comparison with optimized \(F_4\), and automatic
extraction of complete certificates from its traces, remain open.

\appendix
\section{Batch reduction and symbolic closure}
A strict standard representation of an \(S\)-polynomial certifies
that each reducer multiple lies below the pair least common
multiple. We study the size of this certificate in two
representations: an expanded cofactor tuple and an acyclic graph
of monomial shifts, scalar multiplications and binary additions.
The main results are explicit reduced two-element bases for which
the strict bound forces a unique dense cofactor. The shared graph
uses logarithmically many additions in the forced support size.
The lower bounds concern direct bounded pair witnesses and do not
cover specialized algebraic certification rules.

The experiments address a separate question about symbolic closure
and reduction schedules. The exact checker over the rationals
validates the certificate constructions, while the finite-field
prototype measures the cost of reducer discovery and projection.
Neither experiment establishes a runtime advantage over optimized
\(F_4\).

Faug\`ere's $F_4$ algorithm organizes polynomial reductions into batches
and uses sparse linear algebra \cite{F4}. Two operations in such a batch
have different inputs. Symbolic closure discovers the shifted reducers
needed by the selected polynomials. Matrix assembly assigns column
numbers and writes coefficients after those rows have been identified.
We specify both operations, then show how to discard the reducer columns
before echelonizing the selected-pair rows. The contribution is a precise
batch contract linking a schedule-independent closure, a triangular
quotient, and a checkable criterion for the resulting basis updates.

Parallel $F_4$ implementations already exist. Monagan and Pearce
\cite{MP} give a compact parallel implementation, so parallelism alone
does not establish novelty. Berthomieu et al.\ \cite{Ber} address modular
matrix multiplication on GPUs for polynomial-system solving. Boyer,
Dumas and Giorgi \cite{BDG} study exact sparse matrix-vector multiplication
on accelerators and multicore processors. These are relevant backends
for finite-field operations. Our experiment studies the algebraic
interface and batch structure. The triangular block is standard
elimination; the analysis below makes explicit how the chosen closure
makes that block sufficient and how the quotient determines the new
pivots. We make no priority claim for triangular block elimination.

\input{literature}
\input{algebraic_related_work}
\input{foundations}
\section{Symbolic closure with fixed reducer priority}
Let \(K=\mathbb F_p\), let \(K[x_1,\ldots,x_n]\) have a fixed monomial
order, and let \(G=(g_1,\ldots,g_s)\) be a list of nonzero monic
polynomials. The list remains fixed during a batch. Its order gives a
priority among divisors. For a reducible monomial \(m\), choose the first
\(g_i\) with \(\lm(g_i)\mid m\) and define
\[
h_m=(m/\lm(g_i))g_i.
\]
Draw a directed edge from \(m\) to each monomial of
\(\supp(h_m)\setminus\{m\}\). Irreducible monomials have no outgoing
edges. For seed polynomials \(s_1,\ldots,s_b\), let \(T_0\) be the union
of their supports.

The frontier algorithm starts with \(T=F=T_0\) and an empty reducer
dictionary. While \(F\) is nonempty, compute \(h_m\) for each reducible
\(m\in F\) and save it under key \(m\). Collect all its tail monomials
into a candidate buffer. Deduplicate the buffer, subtract \(T\), and
use the result as the next \(F\). Add the new frontier to \(T\).
On completion write \(D\) for the keys of the reducer dictionary and
\(U=T\setminus D\). For deterministic storage, order all monomial keys
and reducer rows by the term order. The prototype implements the frontier
with sets; sorting and set difference give an alternative implementation.

\begin{theorem}[Closure and schedule independence]
The algorithm terminates and returns the least set \(T\) containing
\(T_0\) that contains the tails of every \(h_m\) with reducible \(m\in T\).
It returns exactly one reducer row for each \(m\in D\). Any work-list
schedule that processes every discovered monomial returns the same
sets and rows.
\end{theorem}
\begin{proof}
Each edge strictly decreases the monomial order. From the finite set of
roots \(T_0\), the tree of directed paths has finite branching. If
infinitely many paths existed, repeatedly choosing a child with
infinitely many descendants would give an infinite descending chain.
Thus the reachable set is finite. Each iteration discovers only
reachable vertices, and by induction on path length every reachable
vertex is discovered. A visited set ensures each vertex is processed
once. Reducer priority is fixed, so each row is a function of its key
and the frozen basis.
\end{proof}

For example, in graded reverse lexicographic order with \(x\succ y\),
take \(G=(x^2-y,xy-1)\) and seed \(x^2y\). Priority selects
\(x^2y-y^2\). The next frontier contains \(y^2\), which is irreducible.
If the order of the list is reversed, the chosen row is \(x^2y-x\).
The theorem guarantees independence from scheduling for a fixed priority.
The two priorities can give different residuals.

\section{Quotient elimination and the batch update}\label{sec:quotient}
Let \(V_T\) be the \(K\)-vector space with monomial basis \(T\).
For each seed or reducer row use its coefficient vector in descending
column order. If \(A\) contains these coefficient vectors as rows, then
the coefficient column of a row combination with weights \(v\) is
\(A^\top v\). Consequently \(\ker A^\top\) describes row relations.
The algorithm below uses row elimination.

Define \(W=\spanf_K\{h_m:m\in D\}\), and let \(V_U\) be the coordinate
subspace supported on \(U\). In descending order of \(D\), eliminate
the current coefficient at \(m\) by subtracting that coefficient times
\(h_m\). Denote the resulting linear map by \(Q:V_T\to V_U\).

\begin{theorem}[Triangular quotient and pivot recovery]
One has \(V_T=W\oplus V_U\), with \(Q\) the projection onto \(V_U\)
along \(W\). Put \(L=W+\spanf_K\{s_1,\dots,s_b\}\).
The leading monomials attained by nonzero elements of \(L\) are exactly
the disjoint union of \(D\) and the pivot monomials of an echelon basis
of \(\spanf_K\{Q(s_1),\ldots,Q(s_b)\}\).
\end{theorem}
\begin{proof}
The restriction of the reducer rows to columns \(D\) is triangular
with unit diagonal. Every subtraction clears the current column and
changes only smaller monomials, so it leaves earlier cleared columns
zero. A nonzero combination of reducer rows has leading monomial equal
to the largest reducer key with nonzero coefficient, which proves
\(W\cap V_U=0\). The subtraction algorithm therefore gives the unique
decomposition and \(L=W\oplus Q(L)\). For nonzero \(w\in W\) and
\(q\in Q(L)\), their leading monomials lie in disjoint sets, so the
larger survives in \(w+q\). In an echelon basis the leading monomials
are distinct, and the largest among rows with nonzero coefficient
survives in any combination. This identifies the claimed set of pivots.
\end{proof}

Partitioning the columns into \(D\) and \(U\) gives a block formula.
Write \(H_D,H_U\) for the reducer blocks and \(S_D,S_U\) for the seed
blocks. The residual matrix is
\[
S_U-S_D H_D^{-1}H_U.
\]
The implementation applies the triangular operations directly. It does
not form the inverse. Echelonize the residual rows over \(K\) and append
the nonzero pivots, made monic, to \(G\). Their supports lie in \(U\);
thus every monomial of each such row is irreducible by the frozen
basis. Rows appended together may reduce one another. A final
interreduction may be applied when a reduced basis is required.

\begin{theorem}[Correctness of a pair-driven solver]
Initialize \(G\) with the nonzero monic input polynomials. Process
pending critical pairs in finite batches, using their ordinary
\(S\)-polynomials as seeds. After each batch append the nonzero residual
echelon rows and enqueue all pairs involving each new row. With no pair
omissions, the algorithm terminates and returns a Gr\"obner basis of
the input ideal.
\end{theorem}
\begin{proof}
Every new row is in the old ideal. When a batch appends rows, their
leading monomials are outside the old leading monomial ideal, so that
ideal strictly increases. The ascending chain condition on monomial
ideals gives finitely many augmenting batches, and only finitely many
pairs are ever enqueued.

For each processed pair let \(\lambda\) be the least common multiple
of the original two leading monomials. Its seed \(s\) is zero or has
leading monomial strictly below \(\lambda\). The triangular
subtractions express \(s-Q(s)\) using shifted reducers with leading
monomials at most \(\lm(s)\) when \(s\ne0\).
Express \(Q(s)\) in the appended echelon basis. Since the basis has
distinct leading positions, every participating row has leading
monomial at most \(\lm(Q(s))\), which is at most \(\lm(s)\) if the
residual is nonzero. Hence the processed pair admits a representation
in the final basis whose nonzero summands have leading monomials
strictly below \(\lambda\). This is the standard representation
condition in Buchberger's criterion. The ideal has been preserved, so
the final list is a Gr\"obner basis of the original ideal.
\end{proof}

\section{Sparse assembly and parallel work}
After closure, use a deterministic row list, including seeds. Use
zero-based indexing. If row \(i\) has \(\ell_i\) terms, write
\(a_0=0\), \(a_{i+1}=a_i+\ell_i\), and \(M=a_r\) for \(r\) rows.
Each row emits its terms in \([a_i,a_{i+1})\). Sort the records
\((\text{monomial},\text{original position})\) by descending monomial
order. Mark the first occurrence of each distinct key and scan these
flags to assign column numbers. Scatter each column number to its
original position. The reference preserves the within-row input order.
A backend requiring sorted CSR can sort each segment or emit its terms
in term order before assembly.

\begin{proposition}
The output offsets, column indices and coefficients encode the input
polynomial rows exactly. Row emission and index scattering have
disjoint destinations, with synchronization between stages.
\end{proposition}
\begin{proof}
The intervals partition the occurrence buffer and the sorted record
list contains each original position exactly once. Equal monomials
form one contiguous group, and the flag scan assigns the same column
to all its records and distinct columns to distinct groups. The
scatter changes only the index array, leaving each coefficient at its
original position.
\end{proof}

With fixed-width keys, comparison sorting gives \(O(M\log(M+1)+r)\)
work and \(O(M+r)\) storage. A radix implementation uses
\(O(PM+r)\) work for \(P\) passes with fixed-size radix and linear-work
passes. This avoids repeated merges against the whole dictionary.
For multiword keys, comparison or radix costs must include the word
count. The prototype uses exponent tuples with exact integer arithmetic.
A packed implementation must check that all shifted exponents and
degree fields fit and that the encoding preserves the order on the
actual batch. It must also check that offsets can represent \(M\).
Writing an explicit dictionary of \(N\) monomials takes \(\Omega(N)\)
output records, and writing explicit CSR takes \(\Omega(M+r)\) output
records. These are output bounds.

With the direct reducer search, symbolic closure tests at most
\(|T|s\) divisibilities, each costing \(O(n)\) exponent operations.
Generating reducer supports costs \(O(nE)\) exponent operations for
\(E=\sum_{m\in D}|\supp(h_m)|\). Deduplication and membership have
additional costs. A bulk implementation has one round per nonempty
frontier, bounded by one plus the longest reachable path. Within a
round the reducer searches and shifted-row generation can run
independently. Sparse elimination can create fill within \(T\), and
the cost of the triangular updates depends on the intermediate
supports. The assembly bound does not include that work.

\input{architecture}
\section{Reproducible validation}
The artifact uses Python 3.12.14 and SymPy 1.14.0. It tests cyclic
systems in three and four variables, a three-variable Katsura system,
and twelve seeded random systems of three quadratics in three variables.
All systems are run over primes \(2,101,65521\) in graded reverse
lexicographic order. The reference uses batches of at most sixteen
pairs, prioritized by lcm degree. Exact modular operations include
characteristic two and require no Montgomery representation.

Every batch compares the reducer map and closure set with a sequential
queue, reconstructs every row from the sparse buffers, checks the
reducer columns of the residual are zero, and compares the union of
reducer and residual pivots with ordinary sparse echelonization of the
full row list. Each final basis is checked against SymPy's reduced basis
and its leading ideal. The aggregate counts below sum over all three
fields. Input terms count all reducer and seed occurrences before
projection. Residual terms count only the projected seeds before
their echelonization, so the two columns measure the sizes of different
stages and do not measure speedup.

\begin{center}
\begin{tabular}{lrrrrr}
\toprule
Family & Batches & Pairs & Reducer rows & Input terms & Residual terms\\
\midrule
Cyclic 3 & 6 & 45 & 57 & 282 & 24\\
Cyclic 4 & 21 & 273 & 468 & 3278 & 436\\
Katsura (3 vars.) & 9 & 75 & 220 & 1393 & 116\\
Random quadratics & 184 & 1964 & 5055 & 65888 & 5698\\
\bottomrule
\end{tabular}
\end{center}

The released CSV contains each individual batch. These checks exercise
exactness and quantify the residual workload entering the final
echelonization step. The small systems leave open the scaling of
frontier searches and the cost of the projection. A performance study
would need matched implementations on identical inputs, a competitive
parallel CPU baseline, and separate timings for closure, assembly,
triangular projection, residual elimination, and any host--device
transfers. End-to-end timings would determine whether improvements
in a stage matter to the solver.

\input{extensions}
\input{lazy}
\input{batch_witnesses}

% Historical solver timings retained in the supplement.
\input{matched_section}
\input{audit_section}
\input{lean_verification}
\section*{CRediT authorship contribution statement}
\textbf{Chandrasekhar Gokavarapu:} Conceptualization.
\textbf{Lakshmi Gayathri Isukapalli:} Validation (theoretical results).
\textbf{Janardhana Rao Kolli:} Validation (theoretical results).
\textbf{Fathima Sulthana Md:} Software, Investigation (computational experiments).
\textbf{Sai Durga Rolla:} Software, Investigation (computational experiments).

\section*{Funding}
This research received no specific funding.

\section*{Declaration of competing interest}
The authors declare no competing interests.

\section*{Code and data availability}
The supplementary material contains the Python scripts, the
certificate verifier, the raw experimental CSV files, and instructions
for reproducing the computational results. The Lean project is in
the supplementary folder \texttt{lean}; its \texttt{check\_all.sh}
script and \texttt{verification.log} document the formal checks.

\section*{Preprint statement}
An earlier version of this work is available on arXiv at
\url{https://doi.org/10.48550/arXiv.2601.06765}.The present version has been submitted to arXiv to replace the earlier. 

\section*{Declaration of generative AI and AI-assisted technologies in the manuscript preparation process}
ChatGPT was used during preparation of this work to assist with
LaTeX troubleshooting, language refinement, drafting and revision
of the manuscript, development of the experimental Python code, and
analysis and checking of the computational results, and development
of the Lean formalization. The authors take
responsibility for the content of the manuscript.

\end{document}

%% file: syzygy.tex
\section{Size of strict standard representations}
In the witness graph model, leaves are generators, with scalar
multiplication, monomial shifts and binary addition as operations.
A strict pair witness requires every formal monomial multiple
along the expansion of its root to lie below the pair lcm.
If a formal cofactor must have \(N\) terms, at least
\(\lceil\log_2 N\rceil\) addition gates are required. A leaf
cofactor has at most one term, and each addition can at most
double the largest support obtained so far. For \(g=x-1\),
the recurrence \(P_0=g,\ P_{r+1}=P_r+x^{2^r}P_r\) gives
\(P_r=x^{2^r}-1\) without expanding its cofactor.

\paragraph{Selected pair.}
For \(h=zy+y(x^{2^k}-1)\) in lex order \(z\succ y\succ x\),
the pair \((h,z)\) has seed \(y(x^{2^k}-1)\). With the single
leaf \(g=x-1\) its cofactor is unique and has \(2^k\) terms.
The shifted witness \(yP_k\) has exactly \(k\) additions and
bound \(yx^{2^k}<zy\).

\paragraph{Complete certificate with multiple containment roots.}
In lex order \(z\succ y_m\succ\cdots\succ y_1\succ x\) let
\(G=(g,h)=(x-1,z-y_1)\) and append \(y_i(x^{2^k}-1)\) to
\(G\) to form the input \(F\). The output pair satisfies
\(zg-xh=y_1g-h\), with both multiples below \(xz\).
Each reverse-containment identity \(y_i(x^{2^k}-1)=A_i g+B_i h\)
forces \(A_i|_{z=y_1}=y_i\sum_{r=0}^{2^k-1}x^r\).
The substitution cannot increase support, so expanded cofactors
require at least \(m2^k\) terms. After substitution the same
support argument forces at least \(k\) graph additions.
The distinct targets require \(m\) distinct roots. Sharing \(P_k\)
and shifting it once per target, with constant overhead for the
pair and the immediate containments, gives optimal
\(\Theta(m+k)\) graph size.

\paragraph{Nonminimal basis.}
In lex order \(z\succ w\succ x\), take \(f=x^{2^k}-1\) and
\(G=(g,h,j)=(x-1,z+f,zw)\). This is a basis of \((x-1,z)\).
For \(wh-j=wf=Ag+Bh+Cj\), the strict bound \(zw\) forces
\(C=0\), prohibits \(w\) in \(B\), and prohibits \(z\) in the
coefficient of \(w\) in \(A\). Comparing the coefficient of \(w\)
therefore forces \([w]A=\sum_{r=0}^{2^k-1}x^r\).
The bounded graph uses \(wP_k\); the other pairs have
representations \(-h-x^{2^k}g+g\) and \(-j\).
Thus all pair obligations have optimal graph size \(\Theta(k)\).

\begin{proposition}[Forced cofactor for a reduced basis]
For \(D=2^k\), \(k\ge1\), the list
\(G=(g,h)=(xy-y,x^{D+1}-1)\) in lex order \(x\succ y\)
is a reduced Gr\"obner basis. Its only pair has a unique strict
cofactor tuple with \(D\) terms, and its optimal shared witness
uses \(k\) addition gates.
\end{proposition}
\begin{proof}
The pair lcm is \(x^{D+1}y\), and
\(yh-x^Dg=y(x^D-1)=(\sum_{r=0}^{D-1}x^r)g\) is strictly
bounded. Buchberger's criterion proves the basis property.
Monicity, incomparable leading monomials and irreducible tails
give reducedness. Neither the product nor chain criterion
discards the pair. In any bounded \(Ag+Bh\), every nonconstant
monomial multiplier of \(h\) violates the bound. Thus \(B\)
is constant, and setting \(y=0\) forces \(B=0\). Cancellation
of \(g\) fixes \(A\). The support lower bound is \(k\).
The doubling recurrence starting at this \(g\) yields
\(y(x^D-1)\) with \(k\) additions and bound \(x^Dy\).
\end{proof}

\begin{proposition}[Multivariate forced cofactor]
Use lex order \(x_1\succ\cdots\succ x_n\succ y\), \(n\ge2\),
and let \(D_i=2^{k_i}\), \(k_i\ge1\), \(K=\sum_i k_i\).
The reduced basis
\[
g=y\prod_i(x_i-1),\qquad h=\prod_i(x_i^{D_i}-1)
\]
has a unique strict cofactor tuple for its only pair, with
\(2^K-1\) terms. An optimal bounded graph uses between \(K\)
and \(2K-1\) additions.
\end{proposition}
\begin{proof}
Set \(M=\prod_i x_i^{D_i-1}\) and \(q_i=\sum_{r=0}^{D_i-1}x_i^r\).
Then \(yh-Mg=(\prod_i q_i-M)g\), strictly below the lcm
\(L=y\prod_i x_i^{D_i}\) in the formal bound. The generators
are monic with incomparable leading monomials. A tail term of
\(g\) omits some \(x_i\), while \(h\) is free of \(y\).
Each tail is reduced by both leading monomials.
Buchberger's criterion applies, and neither the product nor
chain criterion discards the pair. In any bounded \(Ag+Bh\),
\(B\) is constant by lex order and vanishes at \(y=0\) in the
identity. Hence \(A=\prod_i q_i-M\), forcing \(2^K-1\) terms
and at least \(K\) additions.

Start with \(U=g,\ V=0\). For each \(i\) and
\(r=0,\ldots,k_i-1\), set \(t=x_i^{2^r}\) and simultaneously
form \(U'=U+tU,\ V'=U+tV\). The cofactor of \(U\) is the
sum over the current monomial block; \(V\) omits its greatest
element. The shifted and unshifted blocks are disjoint and
the greatest element lies in the shifted block. Induction
therefore gives the desired cofactor at the final \(V\) root
with every formal summand below \(L\). Omit the last update of
\(U\) to give at most \(2K-1\) additions.
\end{proof}

For either reduced basis, taking the input to be the basis
makes the containment witnesses immediate. The bounds compare
representations under direct bounded pair verification; they
give no runtime guarantee or lower bound for more expressive
certificate systems. The multivariate input remains sparse
as the degrees grow at fixed \(n\).

%% file: encoding.tex
\section{Encoded length and exact verification}
\label{sec:encoded-length}
We use binary exponent vectors and rational coefficients. Each
integer \(a\) is written by Elias gamma coding of \(|a|+1\),
with a sign bit for signed fields. Its length is
\(\Theta(1+\log_2(|a|+1))\). A
sparse polynomial lists its nonzero terms explicitly, each as an
exponent vector and a reduced rational coefficient. A certificate
graph lists its nodes in topological order. A node stores an
operation tag, its operand indices and, where applicable, its
scalar and shift vector. A certificate specifies the root index
for each obligation. The basis, input and term order are supplied
to both verifiers. The bit counts below refer to the certificate
and exclude this common input.

\begin{proposition}[Expanded length and graph upper bounds]
For the two-element families in the preceding section, take the
input list to be the proposed basis. The bivariate family indexed
by \(k\) requires \(\Omega(k2^k)\) bits for an explicit sparse
certificate and admits a strict graph certificate of
\(O(k^2+k\log(k+1))\) bits. The multivariate family indexed by
\((k_i)_{i=1}^n\), with \(K=\sum_i k_i\), requires
\(\Omega(K2^K)\) bits in the same explicit encoding and admits a
strict graph certificate of \(O(K^2+K\log(K+1))\) bits.
\end{proposition}
\begin{proof}
In the bivariate family, at least half the forced cofactor terms
have exponent at least \(2^{k-1}\), so these terms alone require
\(\Omega(k2^k)\) bits. In the multivariate family at least
\(2^{K-1}-1\) surviving terms have exponent at least \(D_i/2\)
in each given coordinate \(i\). Summing the exponent bits coordinate
by coordinate gives \(\Omega(2^K\sum_i k_i)\), since \(K\ge2\).
The count relies on the explicit term-list representation.

Each graph has \(O(K)\) nodes, with \(K=k\) in the bivariate case.
The shifts in coordinate \(i\) are powers \(2^r\), with at most
two nodes for each such shift. The vector encoding costs
\(O(n+r+1)\) bits per shift. Summing gives
\(O(nK+\sum_i k_i^2+K)=O(K^2+K)\) bits because \(n\le K\).
Scalars have constant size and the operand indices cost
\(O(K\log(K+1))\) bits in total. The root indices and immediate
containment leaves satisfy the same estimate.
\end{proof}

\begin{proposition}[Verification on the constructed graphs]
Let \(S\) be the maximum support size of \(g\) and \(h\) in the
two-element family. A strict certificate with the encoded lengths
above can be verified in
\[
  O\bigl(KS(n+1)(K+1)(1+\log_2(S+1))+K^2\bigr)
\]
bit work using sorted sparse lists and direct access to stored
nodes, where \(n=1\) and \(K=k\) for the bivariate family.
In particular the verification work is polynomial in the bit
length of the sparse input for these families.
\end{proposition}
\begin{proof}
At any doubling stage the full-block value is \(y\) times a
product of binomials. The factors are \(x_i^{a_i}-1\), with
\(a_i\) the current block length in coordinate \(i\). The
top-term-deleted value is this value minus a monomial shift of
\(g\). Therefore their supports, and those of shifted operands
or the pair target, are at most \(2S\). Coefficients in these
lists have bounded bit length. The bivariate construction uses
only shifted copies of two-term values and their sums.

Each exponent has \(O(K+1)\) bits. A monomial shift is computed
entrywise and preserves order. A binary addition merges the
sorted lists and combines coefficients at equal exponents.
The verifier also stores one propagated leading-monomial bound
and a constant-size provenance label per node. The graph
has \(O(K)\) nodes. Sorting the common input and pair target
when necessary costs \(O(S(n+1)(K+1)(1+\log_2(S+1)))\)
bit work by comparison sorting. Charging this cost per node
as an upper bound and including the cost of reading the graph
gives the displayed estimate. For these inputs the sparse
length is at least \(S\) and at least \(K\), since the leading
exponents of \(h\) encode the degree parameters. Also \(n\le K\).
\end{proof}

The shared certificate length is polynomial in \(K\) and hence
in the sparse input length. The expanded lower
bound is exponential in \(K\); its relation to total input length
depends on the number of variables. At fixed \(n\), the input
support is constant and its binary length is linear in \(K\).
The verification estimate applies to the constructed witnesses.
Unrestricted graphs may have intermediate values with large
supports, and the theorem does not imply polynomial-time
verification in compressed graph length for arbitrary instances.

%% file: sharp_length.tex
\subsection{Optimal encoded graph length}
\label{subsec:optimal-encoding}

\begin{lemma}[Cost of a dense interval of subset sums]
\label{lem:dense-bit-cost}
Let a finite list of positive integers \(a_1,\ldots,a_s\) have
subset sums containing every integer in \([0,2^\ell-1]\).
Writing \(C=\sum_j(1+\lfloor\log_2 a_j\rfloor)\), one has
\[
  \ell^2\le4096C.
\]
In particular \(C=\Omega(\ell^2)\), with an absolute constant.
\end{lemma}
\begin{proof}
For an integer threshold \(c\), split a chosen subset into entries
below \(c\) and entries at least \(c\). The former have total
weight at most \(sc\), and hence at most \(sc+1\) possible sums.
If \(b_t\) is the number of entries at least \(2^t\), it follows that
\[
  2^\ell\le (s2^t+1)2^{b_t}.
\]
If \(s+1\le2^r\), then \(s2^t+1\le2^{r+t}\) and therefore
\(\ell\le r+t+b_t\). For any integer \(T\ge0\),
\[
 T(\ell-r-T)_+\le\sum_{t=0}^{T-1}b_t\le C.
\]
The second inequality counts, for each entry \(a_j\), its thresholds
\(2^t\le a_j\); there are exactly
\(1+\lfloor\log_2 a_j\rfloor\) such thresholds.

If \(s\ge\ell^2\), unit cost proves the assertion. Otherwise
\(s+1\le\ell^2\). For \(0<\ell<64\), coverage implies \(C\ge1\),
which gives the displayed bound; the case \(\ell=0\) is immediate.
For \(\ell\ge64\), put \(r=\lfloor\ell/4\rfloor\) and
\(T=\lfloor\ell/2\rfloor\). The elementary inequality
\((4q+4)^2\le2^q\) for \(q\ge16\) follows by induction.
Consequently \(s+1\le\ell^2\le(4r+4)^2\le2^r\).
Also \(\ell\le3T\) and
\(\ell\le5(\ell-r-T)\). Thus
\(\ell^2\le15T(\ell-r-T)\le15C\), which proves the assertion.
\end{proof}

\begin{theorem}[Sharp shared-certificate bit length]
\label{thm:optimal-encoding}
For the bivariate reduced-basis family, the minimum binary length
of a strict certificate in the stated graph encoding is
\(\Theta(k^2)\). For the multivariate family it is
\[
  \Theta\!\left(nK+\sum_{i=1}^n k_i^2\right).
\]
These bounds hold with both basis leaves available and arbitrary
reuse of nodes. For a complete certificate, take the input list
to equal the proposed basis.
\end{theorem}
\begin{proof}
Expand the formal cofactor of \(g\) at the required pair root.
Each term arises from a path to a \(g\)-leaf, with exponent
vector equal to the sum of the shift vectors on that path.
A path uses any shift node at most once. In particular, the
surviving cofactor support is contained in the set of subset
sums of all shift vectors. Cancellations and zero scalars
cannot introduce exponents outside that set.

In the bivariate case, projecting onto the \(x\)-coordinate
gives all integers in \([0,2^k-1]\) in the forced cofactor.
Apply the lemma to the positive entries of that coordinate
in the shift nodes. The integer code charges at least a constant
times the sum in the lemma, giving \(\Omega(k^2)\).
The doubling construction uses \(O(k^2)\) bits, including
operand indices.

In the multivariate case the forced cofactor support is the
box \(\prod_i[0,2^{k_i}-1]\) with its largest vector removed.
Since \(n\ge2\) and every side has length at least two, each
coordinate projection still contains the full interval. Apply
the lemma independently in each coordinate and sum the bit
costs of the disjoint coordinate fields. This gives
\(\Omega(\sum_i k_i^2)\). If there are \(s'\) shift nodes, the
subset-sum count also gives \(2^{s'}\ge 2^K-1\), hence
\(s'\ge K\) since \(K\ge2\). Each shift vector contains
\(n+1\) explicitly encoded coordinates, giving
\(\Omega(nK)\). The two lower bounds imply the sum bound.

The strict block construction uses \(O(K)\) nodes and at
most two shifts with parameter \(2^r\) per coordinate and
stage. Its shift fields use \(O(nK+\sum_i k_i^2)\) bits.
The \(O(K\log(K+1))\) bits for indices are absorbed by this
bound: Cauchy--Schwarz gives \(\sum_i k_i^2\ge K^2/n\), so
\(nK+\sum_i k_i^2\ge 2K^{3/2}\).
The constant number of containment and pair roots does not
change the order.
\end{proof}

The lower bound concerns the explicit binary labels in the
specified witness graph. An encoding with implicit shift
parameters or nodes that multiply by arbitrary polynomials
requires a different argument.

\begin{theorem}[Sharp length with sparse shift labels]
\label{thm:sparse-shift-length}
Replace the full shift vector by a list of its nonzero coordinate
and exponent pairs, in increasing coordinate order. Encode the
list length and positive exponents as above, and each coordinate
index in exactly \(\lceil\log_2(n+1)\rceil\) bits, where \(n+1\)
is the number of variables. Keep the other graph fields unchanged.
For the multivariate family, the minimum length is
\[
 \Theta\!\left(K\log_2(n+1)+\sum_{i=1}^n k_i^2\right).
\]
For the bivariate family it is \(\Theta(k^2)\).
\end{theorem}
\begin{proof}
The path argument of Theorem~\ref{thm:optimal-encoding}
does not depend on how zero entries are stored.
Projecting the forced cofactor onto each coordinate and applying
the subset-sum lemma to the positive shift exponents therefore
still forces \(\Omega(\sum_i k_i^2)\) bits in the exponent fields.
Omit all zero-vector shifts when counting possible subset sums:
if \(s_+\) shift nodes have a nonzero vector, the support has at
most \(2^{s_+}\) elements. The forced support has \(2^K-1\)
elements, hence \(s_+\ge K\). Each such node stores at least one
coordinate index, giving \(\Omega(K\log_2(n+1))\) bits in
separate fields. The bivariate lower bound follows from the
same projection as in Theorem~\ref{thm:optimal-encoding}.

Every shift in the doubling and block constructions has just
one nonzero coordinate. Its coordinate and exponent fields,
over all stages, take \(O(K\log_2(n+1)+\sum_i k_i^2)\) bits.
The length fields, operation tags and constant scalars take
\(O(K)\) bits. The operand indices take \(O(K\log_2(K+1))\)
bits. These fit the claimed bound because
\(\log_2(K+1)\le\log_2(n+1)+O(K/n)\) and
\(K^2/n\le\sum_i k_i^2\).
For the bivariate doubling construction the full-vector upper
bound already gives \(O(k^2)\) in this sparse format.
\end{proof}

For fixed \(n\), omitting zero coordinates retains the quadratic
order in \(K\). If all \(k_i=1\), so that \(K=n\), the
coordinate-by-coordinate model has optimal length \(\Theta(n^2)\)
while the sparse shift format has optimal length \(\Theta(n\log(n+1))\).
These are graph-length statements. In the latter regime, the
explicit input polynomials have \(2^n\) terms each. The claims
do not assert an exponential separation in input bit length as
\(n\) grows.

%% file: shared_exponents.tex
\section{Sharing exponent computations}
\label{sec:shared-exponents}

Consider the bivariate witness model over \(\mathbb Q\) with lexicographic
order \(x\succ y\). An instruction stream interleaves two types of
registers. Integer registers can store a nonnegative binary literal
or the sum of two earlier integer registers. Polynomial registers
can store a generator leaf, the sum of two earlier polynomial
registers, a rational scalar times an earlier polynomial register,
or a monomial shift of an earlier polynomial register. A shift
refers to two integer registers, one for each coordinate.
Each instruction has a fixed-width tag and its operands are encoded
as positive backward distances in the instruction stream, using
Elias gamma codes. Literals are encoded by the integer code of
Section~\ref{sec:encoded-length}. The stream length and the output
register indices use self-delimiting integer codes. In particular,
every instruction costs at least one bit. Strictness requires the
monomial bound before cancellation, as in the literal-shift model.

\begin{theorem}[Sharp length with shared exponents]
\label{thm:shared-exponent-length}
For the bivariate reduced basis
\(G=(xy-y,x^{2^k+1}-1)\), \(k\ge1\), the minimum length of a strict
pair witness in this instruction format is \(\Theta(k)\) bits.
The same estimate holds for the complete certificate when the
input list equals \(G\). The constructed certificate admits
deterministic exact verification in \(O(k^2)\) bit operations.
\end{theorem}
\begin{proof}
Set \(D=2^k\), \(g=xy-y\), and \(h=x^{D+1}-1\).
The pair target is \(yh-x^Dg=y(x^D-1)\) and its bound is
\(L=x^{D+1}y\). In a strict cofactor tuple \((A,B)\), any
nonconstant term of \(B\) would violate the bound. Setting \(y=0\)
therefore gives \(B=0\), so
\(A=\sum_{j=0}^{D-1}x^j\). Along a path from an occurrence of the
\(g\)-leaf to the output, each shift instruction is used at most
once, and each contributes a fixed vector. If there are \(s\) shift
instructions, all exponents of \(A\) lie among at most \(2^s\)
subset sums of these vectors. Addition and rational scaling cannot
create a monomial outside this set. Consequently \(s\ge k\).

Start with integer registers for zero and one and the leaf
\(P_0=g\). At the start of stage \(r\), the stream ends in
\((Z_r,d_r,P_r)\) with \(Z_r=0\), \(d_r=2^r\), and
\(P_r=y(x^{2^r}-1)\). Append five instructions in this order:
\[
 T_r=x^{d_r}y^{Z_r}P_r,\qquad Q_r=P_r+T_r,\qquad
 Z_{r+1}=Z_r+Z_r,\qquad d_{r+1}=d_r+d_r,\qquad
 P_{r+1}=1\cdot Q_r.
\]
Every operand is at backward distance at most five. This yields
\(O(k)\) instruction bits and \(O(\log(k+1))\) bits for the
stream length and output index. In the concrete three-bit opcode
layout of the formal companion, the initial stream occupies
\(14\) bits and each stage occupies \(56\) bits. The instruction
stream therefore has exactly \(14+56k\) bits. With gamma-coded
counts and one output root, the companion constructs a validated
frame of at most \(33+76k\) bits, using a coarse linear bound for
the header fields. The evaluator supplies both generators as leaves
and selects the first leaf in the construction.
The final formal bound is
\(x^Dy<L\), since at stage \(r\) the largest shifted leaf
leading monomial is \(x^{2^r}y\). The identity containment
obligations require only generator leaves. For verification, each
polynomial register has at most two nonzero terms after merging,
and each exponent has \(O(k)\) bits. The exact value and its
formal bound can therefore be checked with \(O(k)\) bit work per
instruction. Forming the pair target and checking the final
identity and strict inequality have the same per-check cost.
\end{proof}

\begin{proposition}[Explicit pair-frame bounds]
\label{prop:explicit-frame-length}
In the concrete three-bit opcode layout and counted frame format
used by the companion, let \(L_k\) be the minimum length of an
accepted strict pair witness for the bivariate family. For \(k\ge1\),
\[
  3k\le L_k\le109k.
\]
\end{proposition}
\begin{proof}
Evaluate a supplied decoded stream while retaining both formal
cofactors and a path count. Each leaf starts with count one.
Addition adds counts, while scalar multiplication and shifts retain
them. A stream with \(N\) instructions therefore has root count
at most \(2^N\). Each surviving monomial of its first cofactor
comes from a formal path. The propagated leading bound also
bounds each surviving cofactor monomial after multiplying by
the corresponding generator's leading monomial.
Strictness forces the second cofactor to vanish, and the pair
identity forces the first cofactor to have \(2^k\) monomials.
Thus \(N\ge k\). The three-bit tags give the lower bound.
The accepted frame constructed above has at most
\(33+76k\le109k\) bits. The Lean companion checks these bounds
on the actual decoded frame and evaluator.
\end{proof}

\paragraph{Exact verification with a support cap.}
Fix a cap \(W\) on the support of each evaluated polynomial
register. For an instruction stream of bit length \(N\) and a
bivariate rational input of bit length \(B\), there is a
deterministic exact verifier for the strict pair identity under
this cap, running in time polynomial in \(N+B+W\). It checks
instruction types and references, evaluates each register in
topological order using sorted term lists, and rejects whenever a
merged result exceeds the cap. In parallel it stores the maximum
shifted leaf leading monomial, using the maximum at addition and
monomial multiplication at shifts. It retains the bound at
scalings, including zero, to implement the convention that
formal paths are checked before cancellation.

The bit sizes in this algorithm are polynomial independently of
the cap. An integer register has \(O(N)\) bits because each
addition increases the maximum bit length by at most one, and
there are at most \(N\) instructions. An exponent in a polynomial
register is an input exponent plus at most \(N\) such shifts.
For coefficients, a common denominator divides the product of
all denominators of input coefficients and scalar instructions.
Indeed each path contributes one input coefficient and uses
each scalar instruction at most once. The product has
\(O(N+B)\) bits, and on this common denominator each path
contributes an integer of bit length \(O(N+B)\). There are at
most \(2^N\) paths counted with operand multiplicity, so the
reduced coefficients also have \(O(N+B)\) bits. Thus each merge,
scalar multiplication and bound comparison takes polynomial
time. The target can be formed directly from the two sparse
input polynomials. The verifier checks exact equality at the
root and compares its propagated bound with the pair lcm.
In particular, any fixed polynomial cap in \(N+B\) gives a
polynomial-time verifier for the resulting restricted format.

\paragraph{Limit of the support argument.}
Small shift programs alone do not bound evaluation support.
Starting with the leaf \(x-1\), repeated updates
\(P_{r+1}=P_r+x^{2^{r+1}}P_r\) produce exactly
\(2^{r+1}\) terms at stage \(r\), since the two supports at
each step are disjoint. The doubling exponents and the polynomial
updates have a linear-length instruction stream by the same
local-reference construction. This example limits the sparse
evaluation algorithm. The result leaves open deterministic
polynomial-time exact verification of unrestricted streams.

%% file: identity_testing.tex
\section{Unrestricted verification and powerful skew identity testing}
\label{sec:verification-pit}

Consider the decision problem of checking a supplied strict witness
for a specified pair of nonzero polynomials over \(\mathbb Q[x,y]\)
in lex order \(x\succ y\), with sparse binary input and the
shared-exponent instructions of Section~\ref{sec:shared-exponents}.
The input supplies the generators, the pair and the output register.
Acceptance requires a well-formed stream, the formal bound strictly
below the pair lcm, and equality with the oriented \(S\)-polynomial.
The verifier checks a witness for the specified pair. The reverse
reduction below uses a fixed reduced basis, so hardness for the
restricted class does not depend on recognizing the basis property.

A powerful skew circuit has addition and multiplication gates and
input gates labelled by integer constants or single-variable powers
with binary exponents. At each multiplication at least one operand
is an input gate. Identity testing asks whether the output polynomial
is zero \cite{KLPowerful}.

\begin{theorem}[Equivalence for strict-pair verification]
\label{thm:verification-pit}
Unrestricted strict-pair verification in the shared-exponent model
is polynomial-time many-one equivalent to polynomial identity testing
for powerful skew circuits over \(\mathbb Z\). The reverse reduction
uses only the fixed reduced basis \((xy-y,x^2-1)\), its unique pair,
and witnesses whose shifts involve only \(y\). In particular,
unrestricted strict-pair verification lies in \(\mathsf{coRP}\).
It admits a deterministic polynomial-time algorithm if and only if
powerful skew identity testing over \(\mathbb Z\) does.
\end{theorem}

\begin{proof}
\emph{From verification to identity testing.}
Decode the instructions and evaluate the integer registers. Each
integer addition increases the maximum bit length by at most one,
so the decoded exponents have polynomial bit length. At a generator
leaf record its leading monomial. At a scalar multiplication retain
the operand's bound, at a shift multiply it by the shift monomial,
and at an addition take the maximum. These rules compute the formal
bound since the order is multiplicative and paths are checked before
cancellation. Each propagated exponent is a leaf exponent plus at
most as many shifts as there are instructions. The formal bound test
therefore takes polynomial time, without expanding any polynomial
register. Map malformed streams and failures of the strict bound test
to the constant nonzero circuit.

For an input passing these checks, build a powerful skew circuit
over the rationals for the difference between the output and the
pair target. A sparse generator is a sum of terms, each built by
starting at its coefficient and multiplying successively by
single-variable powers. A shift also uses successive multiplications
by single-variable powers. Scalar instructions multiply by constant
inputs. The pair target is built from monomial multiples of two
sparse generators and their leading coefficients. The resulting
circuit has polynomial encoding length.

To obtain an integer circuit, put its gates in topological order
with indices \(i\ge1\), and let \(D\) be the product of the positive
denominators of its rational input constants. For each gate of value
\(f_i\), construct the scaled value \(D^i f_i\).
A scaled rational constant is an integer literal. A scaled power
is its power input times the integer literal \(D^i\).
For an addition with operands indexed by \(j,\ell<i\), multiply
the two scaled operands by \(D^{i-j}\) and \(D^{i-\ell}\),
respectively, and add. For a multiplication of operand \(j<i\)
by a rational constant \(c\), multiply the scaled operand by
the integer literal \(D^{i-j}c\). For multiplication by a power
input, multiply the scaled operand by that power and by
\(D^{i-j}\). All multiplications remain skew. The literals have
polynomial bit length since \(i\) and \(\log_2 D\) are polynomially
bounded. The scaled output vanishes exactly when the difference does.

\emph{From identity testing to verification.}
First reduce to univariate powerful skew identity testing.
For a circuit with \(m\) gates whose largest input power exponent
is \(E\) (take \(E\ge1\)), every coordinate degree is at most
\((m+1)E\), since along each path a skew multiplication adds
at most \(E\) to the degree. Substitute the \(j\)-th variable
by \(y^{M^{j-1}}\) with \(M=(m+1)E+1\).
Base-\(M\) uniqueness makes the substitution injective on the
possible monomials, and the new input exponents have polynomial
bit length.

Write \(C(y)\) for the resulting univariate output. Set
\(g=xy-y\), \(h=x^2-1\). Their leading monomials are \(xy\) and
\(x^2\), and their only pair has lcm \(x^2y\).
The identity \(yh-xg=g\) gives a strict representation.
The polynomials are monic and the tails are reduced, so Buchberger's
criterion establishes that \((g,h)\) is a reduced Gr\"obner basis.

Compile the circuit gate by gate into a witness that computes
\(gC(y)\). A constant \(a\) becomes the scaling \(ag\), a power
\(y^e\) becomes a shift of \(g\), addition remains addition, and
a multiplication by a constant or power input becomes a scaling
or shift of the other operand. Each exponent is entered as a binary
literal in the integer stream, with zero for the \(x\)-coordinate.
The resulting stream, including its backward references, has
polynomial bit length. Use \(g+gC(y)\) as the witness output.
Every formal multiple of \(g\) is of the form \(y^b g\) up to
scalar, with \(b\ge0\). Its leading monomial \(xy^{b+1}\)
is strictly below \(x^2y\) in the chosen lex order. Thus
strictness holds for every input circuit. Equality with the
target \(g\) is equivalent to \(C(y)=0\) since \(g\ne0\) and the
polynomial ring is a domain. This proves the reverse reduction.

K\"onig and Lohrey prove that powerful skew identity testing over
\(\mathbb Z\) belongs to \(\mathsf{coRNC}^2\) \cite[Theorem~2]{KLPowerful}.
The forward polynomial-time reduction therefore gives the
\(\mathsf{coRP}\) bound. The two reductions give the equivalence
of deterministic polynomial-time solvability.
\end{proof}

The randomized bound uses identity testing without an intermediate
support bound. The equivalence identifies the decision problem that
must be solved to obtain an unrestricted deterministic verifier.
It concerns verification of a supplied witness, with no assertion
about the complexity of finding or minimizing one.

%% file: literature.tex
\section{Related algorithms and the scope of the contribution}
\label{sec:literature}
\paragraph{From pair reduction to symbolic preprocessing.}
Buchberger's algorithm, as presented by Cox, Little and O'Shea
\cite{CLO}, supplies the critical-pair correctness condition used here.
Faug\`ere's \(F_4\) \cite{F4} shares reduction work across pairs by
constructing batches and generating the reducer multiples before
elimination. The exposition by Bigatti et al.\ \cite{Bigatti} makes
the symbolic preprocessing and order-conversion steps available in
pedagogical implementations. The closure in this article specializes
the preprocessing to a frozen priority and already-cancelled seeds.
It models the possible dependencies without using any of the
coefficients that arise during reduction.

\paragraph{Where sparse reduction already saves work.}
Sparse polynomial reducers can follow the terms currently present
in the accumulator. Roune and Stillman \cite{Roune} study practical
reducer and monomial data structures in both signature and classical
algorithms. Makarim and Stevens \cite{M4} reuse tail-reduced multiples
and work with vectors over irreducible monomials. The existence of
these approaches rules out a general claim that prior methods must
scan every reducer column. In the frozen-basis model, a descending
support-driven reducer using the same priority applies the same
nonzero updates as the active projection. Our comparison makes the
distinction quantitative by measuring repeated divisor searches,
distinct divisor searches, and coefficient updates separately.
Recent work on monomial-ideal data structures \cite{LairezData}
addresses the cost of divisibility queries themselves. Our query
counts are independent of the implementation of such a query.

\paragraph{Linear algebra and storage.}
Faug\`ere and Lachartre \cite{FL} partition the initial matrix into
pivot and nonpivot blocks. GBLA \cite{GBLA} develops layouts and
parallel operations for these matrices. Monagan and Pearce \cite{MP}
describe the Schur complement explicitly and evaluate a parallel C
implementation. Their solver demonstrates why comparing a Python
scan with a Python heap can isolate scheduling effects while leaving
competitive performance open. Prefix scans \cite{Blelloch},
GPU sorting \cite{Satish}, and sparse layouts such as SELL
\cite{SELL} supply established infrastructure. Here they serve as
implementation choices for assembly, while the algebraic results
are stated independently of a device layout. Boyer et al.\ \cite{BDG}
and Berthomieu et al.\ \cite{Ber} give finite-field accelerator
work on sparse matrix-vector and modular matrix-matrix operations,
respectively. These works address numeric kernels rather than the
eager-versus-lazy query comparison considered here.

\paragraph{Current solver context.}
The \texttt{msolve} library \cite{msolve} and \texttt{Groebner.jl}
\cite{Julia} implement modern \(F_4\)-based computation. The latter
also provides tracing. Trace reuse concerns information across
compatible computations. The cache in Section~\ref{sec:lazy} shares
divisor answers across seeds within a single frozen batch.
The REDUCE implementation \cite{ReduceF5}, the \(F_5\) algorithm
\cite{F5}, the signature survey \cite{EF}, and Lairez's axioms
\cite{Lairez} give context for the stronger constraints on labeled
reductions. The example in the foundations section explains why
those constraints cannot be inferred from ordinary row-space
preservation. The framework of Gao, Volny and Wang \cite{GVW}
computes bases for ideals and syzygy modules together. The
certificates here check a single ordinary reduction and do not
compute a syzygy module basis.

\paragraph{Applications and downstream stages.}
Order conversion by FGLM \cite{FGLM}, sparse order conversion
\cite{SparseFGLM}, and the shape-and-stability refinement of
Berthomieu, Neiger and Safey El Din \cite{Shape} take a basis
or quotient data as input. The present projection occurs before
basis completion. Determinantal systems \cite{Det} and sparse
systems \cite{SparseSystems} illustrate how input structure influences
the choice of solving method. Saturation and colon ideals \cite{Sat}
require additional algebraic operations. Rational univariate
representations read from lexicographic bases \cite{RUR} likewise
have a different output. These distinctions matter when interpreting
timings. A faster batch projection alone gives no bound on the cost
of producing a complete representation of the solutions.

\begin{table}[htbp]
\centering\small
\caption{Established methods and the precise claims made here.}
\begin{tabular}{p{.18\linewidth}p{.34\linewidth}p{.37\linewidth}}
\toprule
Topic & Established result or method & Result proved or measured here\\
\midrule
Symbolic closure & \(F_4\) adds shifted reducers for discovered
terms \cite{F4,Bigatti}. & Fixed-priority closure is a finite
reachability problem. Its row set is independent of work-list order.\\
Sparse reduction & Support-driven data structures and reuse of
tail-reduced multiples avoid redundant work \cite{Roune,M4}. &
For a frozen priority, descending direct reduction and active
projection make identical nonzero updates. The count of distinct
lazy queries equals the union of actual inspected supports.\\
Block elimination & The pivot/nonpivot partition and Schur complement
are established \cite{FL,GBLA,MP}. & Closure identifies the triangular
block. Its quotient preserves exactly the remaining leading positions
and gives a standard representation for each processed pair.\\
Reuse and tracing & Implementations share symbolic information
in several ways \cite{msolve,Julia}. & Exact comparison of shared
per-batch cache and eager closure. A cancellation family separates
their query counts by an unbounded factor.\\
Certificates & Pair reductions and signature algorithms can track
representation data \cite{CLO,EF,GVW}. & The frozen-priority identity
certifies the exact residual and its nonzero reducer coefficients
without constructing eager closure. An independent implementation
checks valid and corrupted certificates.\\
Accelerator work & Exact sparse and dense modular kernels have been
implemented \cite{BDG,Ber}. & Matched CPU measurements of eager and lazy reduction, and a specification
of the data and synchronization required at the batch boundary.\\
\bottomrule
\end{tabular}
\end{table}

\paragraph{What the closest reducers reuse.}
For a frozen basis, support-driven ordinary reduction chooses
a term currently present, finds a divisor and applies its shifted
row. With descending selection and the same priority its nonzero
updates coincide with ours. Caching the divisor answer removes
repeated searches. The cached row still has its original tail,
so applying it can expose another reducible monomial. M4GB
\cite{M4} stores tail-reduced multiples and reuses them in subsequent
reductions. Applying such a row can bypass intermediate reducible
monomials altogether. Its work therefore cannot be inferred from
our count of one update per visited reducer key. Classical
\(F_4\) \cite{F4} can also reuse previously reduced rows through
its simplification step, whereas the cache studied here expires
with the frozen batch. The distinguishing claims of this article
are the exact count for this cache, the cancellation separation
between two specified discovery schedules, and the certificate
theorem. No priority claim is made for caching or sparse term
selection.

%% file: algebraic_related_work.tex
\subsection{Related algorithms and certification methods}
Sparse arithmetic, heap division, geobuckets and packed monomial
encodings precede the data structures used here
\cite{Johnson74,MP07,MP11,Yan98,Bachmann98}. Pair criteria and sugar
strategies \cite{Gebauer88,Giovini91} change which pairs are processed;
our comparison fixes the batch. Classical basis and syzygy theory
\cite{BW93} supplies the generation argument. Maletzky and Immler
formalize \(F_4\) and executable syzygy computation in Isabelle/HOL
\cite{Maletzky18}. Our checker verifies individual outputs by identities
and has no machine-checked implementation proof.
Gr\"obner tracing \cite{Traverso88,Joux11,Demin26} and modular
verification \cite{Arnold03} provide precedents for reuse and output
checking. Here the comparison concerns reuse within one fixed-field
batch. Hofstadler and Verron study sparsest cofactor
representations for ideal membership, with size measured by the total
number of terms in the explicit cofactors \cite{HV24}. Our lower
bounds concern the stricter task of representing an \(S\)-polynomial
with each formal reducer multiple strictly below the pair least
common multiple. For the reduced-basis families studied below, this
condition forces the cofactor tuple, so the support lower bound holds
for every explicit strict representation. We compare that support
measure with the number of binary additions in a shared graph of
monomial shifts, scalar multiplications and sums. The separation
depends on the strict verification rule and the graph model. It gives
no general lower bound on the runtime of ideal-membership algorithms
or on verification with more expressive certificates.

%% file: foundations.tex
\section{Algebraic foundations and the computational pipeline}
Let \(R=K[x_1,\ldots,x_n]\), with \(K=\mathbb F_p\), and fix an
admissible monomial order. For nonzero \(f\), write
\(\operatorname{lt}(f)=\operatorname{lc}(f)\lm(f)\).
For nonzero \(f,g\) set \(\lambda=\operatorname{lcm}(\lm(f),\lm(g))\).
Their \(S\)-polynomial is
\[
 S(f,g)=\frac{\lambda}{\operatorname{lt}(f)}f-
        \frac{\lambda}{\operatorname{lt}(g)}g.
\]
The two contributions at \(\lambda\) cancel. A standard representation
of this pair is an expression for \(S(f,g)\) in basis multiples whose
individual leading monomials lie strictly below \(\lambda\). Such
representations for all pairs imply Buchberger's criterion.

\paragraph{The stages of a batch.}
Pair selection fixes the seeds. Symbolic preprocessing finds reducer
multiples and repeats this search on newly discovered terms. Assembly
assigns column indices to the resulting supports. Linear algebra
computes the pivot rows, and the basis update adds the new polynomials
and generates further pairs. A reduction in assembly time can benefit
the full solver only in proportion to the time spent assembling
matrices. The share is workload dependent. It must be measured along
with the closure and elimination stages.

\paragraph{Relations and module data.}
Let \(F=R^m\) and let \(\varphi:F\to R\) send \(e_i\) to the \(i\)-th
input polynomial. The kernel records relations among the inputs. For
a batch with rows represented by \(u_1,\ldots,u_r\in F\), a left-null
vector \(v\) of its coefficient matrix gives the relation
\(\sum_i v_i u_i\in\ker\varphi\). This observation has a different
output from echelonization. For example, over any field the one-row
matrices \((1,0)\) and \((0,1)\) both have trivial left nullspace but
have different pivot columns. A list of row relations alone therefore
does not determine the new leading monomials.

\paragraph{Signatures.}
With a fixed module term order, the leading module monomial of a
chosen representative gives its signature. A regular signature update
uses a shifted reducer with strictly smaller signature than the target.
Filtering an unlabeled list of rows cannot enforce that condition.
For instance, label \(x+y\) and \(x\) by \(e_1\) and \(e_2\), with
\(e_1\prec e_2\). Subtracting the second row from the first gives the
representative \(e_1-e_2\), whose leading module monomial is \(e_2\).
Our solver uses ordinary pairs. The coefficient identity for relations
remains valid when module representatives are retained.

%% file: architecture.tex
\section{Parallel implementation and finite-field arithmetic}
The original motivation is to identify which parts of the batch can
execute in bulk. Within one frontier, monomial divisibility queries and
shifted support generation are independent. Their outputs can be
placed in disjoint ranges by a count and scan. The next frontier
requires a global deduplication and set difference against the visited
keys. A long dependency chain gives many rounds even when individual
rounds are inexpensive. The observed time of a device implementation
will depend on both frontier widths and synchronization overhead.

\paragraph{Layout and the boundary with elimination.}
Row lengths can vary widely. A practical implementation can group
rows by length for launch scheduling while retaining the logical row
order in the offsets. The CSR buffers describe the initial rows.
A pivot may create new positions inside the fixed monomial dictionary,
so initial offsets cannot be used as bounds on intermediate fill.
Similarly, a row's leading column can move between panels. Over
\(\mathbb F_2\), subtracting \((1,1)\) from \((1,0)\) produces
\((0,1)\). Any panel scheduler must admit that row to the later panel.
Projection with frozen reducers admits independent work across seeds.
Echelonization of the residuals still needs a pivot policy and a
representation for intermediate fill.

\paragraph{Exact modular updates.}
For \(0\le a,b,c<p\), the integer \(a+bc\) is at most \(p(p-1)\).
Thus \(p<2^w\) permits a double-width accumulator for a single
multiply-add. To combine \(k\) products before reduction, check
\((p-1)+k(p-1)^2<2^{2w}\); this bound concerns nonnegative sums.
A subtraction implementation must separately bound signed
intermediates or use modular representatives.

\paragraph{Barrett and Montgomery options.}
For \(0\le x<2^{2w}\), set \(\mu=\lfloor 2^{2w}/p\rfloor\) and
\(q=\lfloor x\mu/2^{2w}\rfloor\). The approximation to \(x/p\)
has error less than one before taking the floor. Hence \(q\) is
either \(\lfloor x/p\rfloor\) or one less, and one conditional
subtraction reduces \(x-qp\). A backend must compute the requisite
high part of the product without dropping carries.
For odd \(p<2^w\), put \(R_0=2^w\). With Montgomery operands,
let \(t\) be their product and
\(m=(-tp^{-1})\bmod R_0\). Then \(t+mp\) is divisible by \(R_0\)
and its quotient is below \(2p\). A full-width implementation
must retain the carry from \(t+mp\), which can exceed
\(2^{2w}-1\). In characteristic two, addition and multiplication
of scalar residues can instead use XOR and AND. The prototype
uses exact integer arithmetic for all primes.

\paragraph{Measurements needed for an accelerator implementation.}
A matched device study should compare the same frozen batches with
the same reducer priority and coefficient field. Report frontier
round counts, transfer and launch costs, closure and assembly times,
projection and residual elimination, and total solver time. Compare
output coefficients and pivot positions before comparing speed.
The present CPU data quantify batch structure and the work saved by
an active schedule under the stated conditions.

%% file: extensions.tex
\section{A worked batch from a critical pair}
Work over \(\mathbb F_5\) in graded reverse lexicographic order with
\(x\succ y\). Use \(g_1=x^2-y,\ g_2=xy-1,\ g_3=y^2-x+1\).
Select the pair \((g_1,g_2)\). Its lcm is \(x^2y\), and the seed
\(s=yg_1-xg_2=x-y^2\) has initial support \(\{y^2,x\}\).
The reducer for \(y^2\) is \(g_3\). Its constant term is new,
so the second frontier is \(\{1\}\). Neither \(x\) nor \(1\) admits
a reducer. Thus \(D=\{y^2\}\), \(U=\{x,1\}\), and the rows in
column order \(y^2,x,1\) are
\[
 \begin{array}{c|rrr}
 &y^2&x&1\\\hline
 g_3&1&-1&1\\
 s&-1&1&0
 \end{array}
 \qquad
 Q(s)=s+g_3=1.
\]
The CSR offsets are \((0,3,5)\), the column indices are
\((0,1,2,0,1)\), and the values are\linebreak \((1,4,1,4,1)\) if the rows
are emitted in descending order. The full matrix has pivot columns
\(y^2\) and \(1\). The residual matrix has one pivot, at \(1\), so
the batch inserts the unit polynomial. The processed pair has the
representation \(s=1-g_3\) with both leading monomials below \(x^2y\).
This example also distinguishes closure from matrix assembly.
A dictionary built only from the seed would omit the constant column.

\section{Active reducer scheduling}
The projection in Section~\ref{sec:quotient} scans all reducer keys for every seed.
When the reducer block is large and the current seed touches few of its
columns, a priority queue can avoid most zero-coefficient tests.
Precompute the descending order of \(D\). For each seed, insert its
currently supported reducer keys into a max-priority queue. Remove the
largest key and check its current coefficient. If the coefficient is
nonzero, apply the corresponding reducer and insert every reducer key
newly supported by that update that has never been queued. A key whose
coefficient vanished while it was pending is skipped on removal.
Each seed owns its queue and coefficient map.

\begin{theorem}[Exact active scheduling]
The queue algorithm gives the same residual as the descending scan.
If \(a_j\) distinct reducer keys are queued for seed \(j\) and \(w_j\)
is the sum of the lengths of reducer rows actually applied, the
per-seed work is
\[
 O\bigl(|\supp(s_j)|+w_j+a_j\log(a_j+1)\bigr)
\]
under constant-time expected dictionary access and comparison of
precomputed integer ranks. The shared rank table takes
\(O(|D|\log(|D|+1))\) work to construct by comparison sorting.
\end{theorem}
\begin{proof}
Immediately before each extraction, all nonzero reducible columns
are pending. This holds at initialization. Extracting the largest
pending key, including when its coefficient has become zero, can
introduce only smaller keys. All nonzero keys introduced by an
actual subtraction are either still pending or newly inserted.
Previously extracted keys are larger than the current key and cannot
reappear. Thus the invariant holds and the extractions perform the
same nonzero updates in the same order as the scan. Each key enters
at most once. Heap construction and updates contribute the stated
heap term, and the scans of changed coefficients are bounded by
the lengths of the applied reducer rows.
\end{proof}

In a batch with \(b\) seeds, the scan incurs \(b|D|\) coefficient
tests regardless of sparsity. The active schedule is advantageous
in work when the queue overhead is smaller than the skipped tests.
This is a conditional bound, since fill-in can drive \(a_j\) close
to \(|D|\). Closure still uses the full reachable set. The active
schedule saves projection work when cancellations or seed-specific
dependencies leave much of that set unused.

\section{Relation to previous block algorithms}
Faug\`ere and Lachartre \cite{FL} partition Gr\"obner matrices into
pivot and nonpivot blocks and use specialized block operations and
storage. Monagan and Pearce \cite{MP} explicitly describe the same
Schur complement used in Section~\ref{sec:quotient} and compare a parallel C
implementation. The block formula is therefore prior work. Ordinary
reducer closure is also part of \(F_4\). The contribution here is
the combination of a fixed-priority reachability specification,
a schedule-independent closure proof, and an exact active-work bound
for projecting each seed. The queue specialization avoids zero tests
by using the strictly descending reducer dependencies. It is a
reference algorithm for analyzing and testing the symbolic and
numeric interface. Sparse pivot-driven elimination can exploit
the same dependencies, and the active scheduling argument alone
does not establish priority over all such implementations.

The implementation uses already-cancelled \(S\)-polynomials as seeds.
Some implementations start from the two shifted pair rows. Their
matrices can therefore differ in shape even on the same pair list.
The standard-representation proof applies to the seed convention
used here. All timing comparisons below use that convention on
both sides.

\paragraph{A family separating the two schedules.}
In two variables take the frozen basis \(G=(x-1)\) and seeds
\(x^d y^i\), \(1\le i\le b\). For fixed \(d\ge1\), closure
contains exactly \(bd\) reducer keys \(x^j y^i\) with
\(1\le j\le d\). Each seed traverses only its own chain, visiting
exactly \(d\) keys and yielding \(y^i\). The scan performs
\(b^2d\) coefficient tests. With ranks prepared, the active
algorithm performs \(bd\) extractions from heaps of size one and
\(bd\) two-term updates. Including shared rank sorting gives
\(O(bd\log(bd+1))\) expected work in the stated model. This family
isolates a structural difference in the projection. The seeds are
synthetic and need not arise together as critical pairs.

%% file: lazy.tex
\section{Cancellation-aware closure and a lazy batch certificate}
\label{sec:lazy}
The reachable set records every possible dependency of the chosen
reducers. Actual coefficients can cancel before a reachable term
becomes active. A batch can therefore discover reducer information
on demand. This section gives an exact comparison between eager
closure and a lazy cache, including the information needed to verify
a residual. All statements concern a single frozen basis and its
fixed reducer priority.

For each seed maintain a polynomial \(f\) and a set of inspected
monomials. Repeatedly choose the greatest supported monomial that
has yet to be inspected. Consult a shared cache keyed by monomial.
On a miss, perform the fixed-priority divisibility search and store
either the selected reducer or a negative result. When the cache
returns \(h_m\), record the coefficient of \(m\) in \(f\), and
subtract that coefficient times \(h_m\). When the result is negative,
retain the term. The next choice uses the new support. The cache may
be shared across seeds because the answer to every query depends only
on the frozen basis and its priority. A cache entry is invalidated when
the basis or priority changes.

\begin{theorem}[Lazy discovery and exact batch certificates]
\label{thm:lazy}
Let \(T\) be the eager reachable set and \(D\) its reducible subset.
For each seed \(s_j\), lazy discovery returns the same residual \(q_j\)
as eager triangular projection. Its inspected set \(C_j\) is contained
in \(T\). Across the batch the cache makes exactly
\(\lvert\bigcup_j C_j\rvert\) distinct reducer queries. Moreover,
\(C_j\) consists precisely of the reducer keys used with nonzero
coefficient and the support of \(q_j\).
\end{theorem}
\begin{proof}
The maximum inspected monomial strictly decreases. Once inspected, a
reducible monomial has coefficient zero, and an irreducible monomial
can be left unchanged. A later subtraction has a strictly smaller
leading monomial and cannot alter either coefficient. Every monomial
in an intermediate support belongs to \(T\), since all tails of
selected reducers belong to \(T\). Finiteness of \(T\) now implies
termination.

The difference \(s_j-q_j\) is a linear combination of the selected
\(h_m\), and \(q_j\) is supported on irreducible monomials. Uniqueness
of the decomposition \(V_T=W\oplus V_U\) identifies \(q_j\) with the
eager projection. Each inspected reducible term causes one nonzero
update. Every inspected irreducible term survives to the residual,
and every residual term was inspected before termination. This
identifies \(C_j\). The cache uses one query for each distinct
monomial in the union, since hits reuse both positive and negative
answers.
\end{proof}

\paragraph{What a verifier checks.}
A certificate consists of the residual and the nonzero coefficients
\(\gamma_{jm}\) of the reducer multiples applied to each seed. For
each listed \(m\), the verifier checks that the row is the monic
multiple of the first basis polynomial whose leading monomial divides
\(m\). It checks that the residual contains only irreducible terms
and verifies the exact coefficient identity
\[
 s_j-q_j=\sum_m\gamma_{jm}h_m.
\]
If every listed coefficient is nonzero and each key is listed at most
once, the identity also forces each listed key to be reachable.
Indeed, if some key outside \(T\) were listed, the largest such key
could not be canceled by smaller keys. Nor could it be canceled by a
larger key in \(T\), whose entire row is supported in \(T\). Its
coefficient on the left is zero because \(s_j\) belongs to \(V_T\)
and the residual is irreducible. More explicitly, the offending key
is reducible, so it cannot occur in the residual. This contradiction
places every listed key in \(D\), and the identity then places
\(q_j\) in \(V_U\). Uniqueness of the triangular decomposition
verifies both the residual and all listed coefficients. Thus a
certificate can be checked without reproducing the eager closure or
the schedule that produced it.

\paragraph{An arbitrarily large gap in reducer discovery.}
Over any field, use one variable \(x\), frozen priority
\((x-1)\), and seed \(x^d-x^{d-1}\). Eager closure contains
\(x^d,x^{d-1},\ldots,1\) and queries all \(d+1\) keys.
The lazy algorithm applies exactly one multiple and returns zero.
It inspects just \(x^d\), since the coefficient of \(x^{d-1}\)
cancels in that update. The certificate has one nonzero entry.
This is a separation in symbolic discovery distinct from the
disjoint-chain separation for projection. It compares specified
schedules, and the seed is synthetic.

\paragraph{Trade-off.}
The eager computation exposes independent work on each frontier and
shares the closed reducer dictionary before elimination. Lazy discovery
uses coefficient information earlier. A shared cache saves repeated
searches but can require synchronization in a parallel implementation.
The query count does not include the cost of choosing the greatest
currently supported term. The reference replay uses a simple maximum
search and reports counts rather than attributing speed to that
implementation.

%% file: batch_witnesses.tex
\section{Bounded syzygy witnesses from selected batches}
For a monic list \(G\), let \(\phi_G(e_i)=g_i\). For pair \(p\)
write \(a_p\) for the two-term module lift of \(S_p\) and
\(\lambda_p\) for its lcm. Label each shifted reducer \(u g_i\)
by \(u e_i\); row operations act on labels as well as polynomials.
A relation involving several seeds need not verify each pair.

\begin{lemma}
Suppose \(S_p=\sum_q c_{pq}S_q\), with
\(S_q=\phi_G(w_q)\) for each used \(q\). If every monomial
multiple in each used \(w_q\) has leading monomial below
\(\lambda_p\), then \(a_p-\sum_q c_{pq}w_q\) is a syzygy with
Schreyer leading term equal to that of \(a_p\). If these
witnesses cover all pairs, \(G\) is a Gr\"obner basis of its
ideal and the witnesses generate \(\ker\phi_G\).
\end{lemma}
\begin{proof}
The image of the difference is zero. All terms contributed by
the representations are strictly below the two pair terms in
the Schreyer order. The representation satisfies Buchberger's
criterion, and the leading term assertion gives generation by
Schreyer's theorem \cite{BW93}.
\end{proof}

For lists written as row matrices, verifying \(F=GA\) and
\(G=FB\) proves equality of the two ideals. The columns of
\(I-BA\) and the images under \(B\) of generators of
\(\ker G\) generate \(\ker F\): both sets map to zero, and
\(Fu=0\) gives \(Au\in\ker G\) and
\(u=(I-BA)u+B(Au)\). The supplementary rational checker
verifies both ideal containments and bounded representations
of every pair by evaluating a shared graph of sums and monomial
shifts. A bound on formal summands propagates through that graph
even when their polynomial values cancel. The checker verifies
supplied witnesses; automatic extraction from the finite-field
solver remains open.

\subsection{Separation from caching complete reduction states}
Over \(\mathbb Q\) with lex order \(z>y_1>\cdots>y_m>x\), take
the frozen list \(x-1,z,h_1,\ldots,h_m\), where
\(h_i=zy_i+x^d-1+i(x-1)\), \(d\ge3\), \(m\ge2\).
This is a Gr\"obner basis of \((x-1,z)\), since the first two
elements already form a basis and the others belong to the ideal.
The actual pairs \((h_i,z)\) have lcms \(zy_i\) and seeds
\(s_i=x^d-1+i(x-1)\). They satisfy
\(s_i=(2-i)s_1+(i-1)s_2\).
Reducing the two spanning seeds by \(x-1\) costs \(2d\)
monomial subtractions. Each remaining witness can use two
scalar references, and all reducer multiples lie below each
of the displayed lcms.

A sparse reducer that shares shifted rows and caches completed
full states, applying one monomial multiple at a time, encounters
\(x^k-1+i(x-1)\) for every \(2\le k\le d\). These states are monic
and distinct across \(i,k\), so scalar normalization does not enable
reuse. This costs at least \(m(d-1)\) subtractions before degree one.
The explicit dependence formula uses constant scalar arithmetic per
pair, excluding coefficient bit costs. This operation-count separation
concerns the selected pairs, which include redundant generators.
A baseline that separates the common high-degree part or performs
the same rank computation can recover the saving. Neither total solver
speed nor an advantage over optimized \(F_4\) follows.

\paragraph{Shifted and mixed witnesses.}
In the same ring and order, write \(f=x^d-1=qg\), where
\(q=\sum_{r=0}^{d-1}x^r\) and \(g=x-1\).
For \(b\in\{0,1\}\) choose \(h_i=zy_i+y_i(f+big)\).
The pairs \((h_i,z)\) have seeds \(y_i(f+big)\) and lcms \(zy_i\).
Their supports are disjoint across \(i\), so scalar rank keeps all
\(m\) seeds. With \(b=0\), sharing the shifted witness \(f=qg\)
uses one \(d\)-step reduction instead of \(md\) subtractions
under scalar-normalized complete-state caching.
With \(b=1\), even allowing monomial shifts of whole states does
not identify the states above degree one. Removing \(y_i\) leaves
the monic polynomial \(x^r-1+i(x-1)\), \(2\le r\le d\), whose
nonzero constant term prevents further monomial factorization.
The degree and coefficient of \(x\) distinguish these states.
The baseline makes at least \(m(d-1)\) subtractions. The shared
identity \(y_i(q+i)g\) uses one \(d\)-step reduction and two
references per pair. Its multiples lie below \(zy_i\). The
corresponding class of syzygies has representatives
\(e_{h_i}-y_i e_z-y_i(q+i)e_g\) for \(b=1\).
A baseline that factors monomial content and shares partial linear
combinations can recover the saving. These are operation counts
for selected pairs and exclude the costs of finding dependencies
and certifying the remaining pairs.

%% file: matched_section.tex
\section{Matched-batch replay}
The timing trials compare full runs, which can diverge when the full
matrix supplies different polynomial representatives. To isolate
the effect of symbolic discovery, we replay each frozen batch with
four schedules. The first performs a descending scan of all reducer
columns. The second uses the active queue with the same eager closure.
The third recomputes the first dividing basis row on demand without
a cache. The fourth caches those queries across all seeds in the
batch. All four methods use identical polynomial seeds, field, term
order and reducer priority. The residuals agree coefficient by
coefficient before any echelonization is performed.

The experiments use Cyclic 3, 4 and 5 over each of
\(\mathbb F_2,\mathbb F_{101},\mathbb F_{65521}\). The equations are
\[\sum_{i=1}^n\prod_{j=0}^{d-1}x_{1+(i+j-1)\bmod n}=0\]
for
\(1\le d<n\), together with \(\prod_i x_i-1=0\).
At most sixteen pairs are selected per batch using lcm degree.
The next batch is produced by echelonizing the verified common
residuals and appending monic pivot rows. The raw CSV identifies
each batch by system size, field and index.

\begin{center}
\small
\begin{tabular}{rrrrrrr}
\toprule
\(n\)&Batches&Seeds&Eager queries&Lazy queries&Direct queries&Updates\\
\midrule
3&6&45&96&78&147&123\\
4&21&273&738&557&2084&1648\\
5&338&5371&37052&31179&201976&195436\\
\bottomrule
\end{tabular}
\end{center}

A query is a search of the frozen basis for the first leading
monomial dividing the requested monomial. Eager queries count
\(|T|\); lazy queries count cache misses; direct queries count
all inspected monomials with multiplicity across seeds. Updates
count actual nonzero reducer applications, recorded on the direct path. The trace comparison in
Section~\ref{sec:auditexp} independently tests update equality. In 356 of the 365 batches the lazy count
is strictly smaller than the eager count. The full scan makes
630, 6689 and 491502 coefficient tests for the three sizes,
respectively, summed across the fields. These tests and divisor
queries have different costs. Neither count measures wall time.

\paragraph{Interpreting the measurements.}
The direct reducer is a stronger work baseline than a full scan
when the support is small, since it skips inactive coefficients.
The advantage of caching relative to the direct reducer is reuse
of the divisibility answer. The advantage of laziness relative to
eager closure is that a cancellation may suppress some queries
entirely. Neither mechanism reduces the number of nonzero
updates under the same priority. A tuned solver may combine
both mechanisms with faster divisor indexing or with reuse of
fully reduced multiples. The results therefore describe these
mechanisms under matched conditions and do not rank solvers.

\paragraph{Reproduction.}
The Python scripts use exact residues and include assertions for
row reconstruction, matched residuals and pivot recovery.
The first validation script uses SymPy as an independent algebra
reference on its 45 system--field instances. The timing scripts
include their raw measurements and instructions for running
locally or in a CPU Colab runtime. Repeated timing results depend
on the environment, whereas the operation counts are determined
by the input and fixed priority.

%% file: audit_section.tex
\section{Timed matched replay and certificate validation}
\label{sec:auditexp}
The expanded replay uses the same frozen basis and seeds for
all schedules. It includes Cyclic 3--5, the three-variable Katsura
system, and the twelve random quadratic systems from the
correctness study. Each runs over the same three prime fields.
The eager scan uses descending reducer keys. Eager active,
direct, and lazy use a heap ordered by decreasing monomial.
Direct recomputes a divisor on each inspection; lazy shares
positive and negative answers across seeds. The heap admits each
key at most once per seed. A key canceled before its extraction
is skipped. All paths emit the ordered sequence of nonzero
updates \((m,\gamma_m)\). The test compares both that sequence
and the residual coefficient dictionary on every seed.

The verifier recomputes the first divisor independently,
constructs its shift in the coefficient identity and checks
irreducibility of the residual. It receives no closure or heap.
Five negative tests exercise duplicate keys, use of a lower
priority divisor, a changed coefficient, a reducible residual,
and a missing update. The wrong-priority example would satisfy
the coefficient identity if the second available divisor were
allowed. The certificate validates reduction relative to the
given frozen basis. Verifying the complete solver would require
checking basis membership and all final critical pairs.

All timing columns in the first table are in milliseconds. The second
reports valid certificates (one per seed), their total size in trace
pairs and residual terms, and batches where lazy and eager query counts tie.
\input{audit_table}

The timing run uses Python 3.12.14 on one execution thread in an AMD EPYC 9V74 environment. Method order
rotates by batch. Each batch is measured once per method. The
eager methods are charged the shared measured closure time
separately. The clocked components include heap selection and
maintenance, cache operations, divisor search with row construction,
and coefficient updates with trace emission. Eager divisor
searches are included in the closure measurement. Totals also
include loop and instrumentation overhead, and exclude verification
and residual echelonization. Verification has its own clock.
These are instrumented prototype measurements. They compare
identical batches and do not compare against M4GB or another
external solver.

Across 558 batches, all 7,728 certificates passed, with 229,670
nonzero trace entries. Lazy discovery matched the eager query
count in 74 batches and used fewer queries in 484. In the reported run, lazy reduction was faster in aggregate than eager scan on the Cyclic and random groups. For Katsura, the totals were close (5.92 ms for lazy and 5.42 ms for eager scan). A separate rerun reversed the Katsura ordering (5.41 ms and 5.49 ms, respectively).
The heap and cache overhead therefore matter even when the
query count improves. Independent verification took longer
than generation on each aggregate group. In this reference
verifier, each recorded term is checked against the frozen
basis again, without sharing divisor answers across seeds.
The certificate provides a separate correctness check whose
cost need not be paid during ordinary solving.

%% file: audit_table.tex
\begin{center}\small
\resizebox{\linewidth}{!}{\begin{tabular}{llrrrrrrr}\toprule
Family & Method & Queries & Total & Closure & Heap/scan & Cache & Search & Update\\
\midrule
Cyclic 3--5 & scan & 37886 & 7758.19 & 6577.96 & 503.10 & 0.00 & 0.00 & 523.27\\
 & active & 37886 & 8099.77 & 6577.96 & 829.81 & 0.00 & 0.00 & 534.33\\
 & direct & 204207 & 33045.85 & 0.00 & 906.40 & 74.83 & 31221.05 & 587.77\\
 & lazy & 31814 & 5637.43 & 0.00 & 860.40 & 51.96 & 3978.88 & 543.23\\
Katsura & scan & 260 & 5.42 & 2.95 & 1.23 & 0.00 & 0.00 & 0.74\\
 & active & 260 & 5.69 & 2.95 & 1.52 & 0.00 & 0.00 & 0.78\\
 & direct & 1005 & 14.62 & 0.00 & 1.65 & 0.14 & 11.36 & 0.76\\
 & lazy & 246 & 5.92 & 0.00 & 1.83 & 0.17 & 2.52 & 0.78\\
Random & scan & 6735 & 298.81 & 198.92 & 38.39 & 0.00 & 0.00 & 45.37\\
 & active & 6735 & 341.11 & 198.92 & 76.73 & 0.00 & 0.00 & 46.35\\
 & direct & 37272 & 1105.98 & 0.00 & 80.72 & 7.57 & 936.00 & 51.24\\
 & lazy & 5903 & 282.38 & 0.00 & 77.66 & 6.89 & 125.08 & 47.53\\
\bottomrule\end{tabular}}\end{center}
\begin{center}\small
\resizebox{\linewidth}{!}{\begin{tabular}{lrrrrrr}\toprule
Family & Batches & Seeds & Trace pairs & Residual terms & Verify (ms) & Equal queries\\
\midrule
Cyclic 3--5 & 365 & 5689 & 197207 & 7000 & 31032.44 & 9\\
Katsura & 9 & 75 & 889 & 116 & 12.72 & 3\\
Random & 184 & 1964 & 31574 & 5698 & 989.91 & 62\\
\bottomrule\end{tabular}}\end{center}

%% file: lean_verification.tex
\section{Formal verification companion}
\label{sec:formal-verification}
The supplementary Lean~4 project uses Lean~4.19.0 and Mathlib at
commit \texttt{c44e0c8e\allowbreak e63ca166\allowbreak450922a3\allowbreak73c7409c\allowbreak5d26b00b}.
The project contains kernel-checked components of the algebraic
and encoding arguments. Each source module is checked with warnings
treated as errors. The sources contain no \texttt{sorry} placeholders
and introduce no additional axioms.

For the fixed pair \((xy-y,x^2-1)\), the companion proves monicity
for the explicit lexicographic leading exponents, standardness of
the tails, and leading divisibility for every nonzero member of
the generated ideal. It checks the reverse translation on shared
register programs and the formal bound on its paths. The forward
components include denominator clearing, integrality of normalized
literals, shared-register rescaling, and an operation-count bound
for the rescaling blocks. The Kronecker components prove the
uniform coordinate bound, preservation of the zero test and a
binary magnitude bound.

The binary components check gamma fields, opcode decoding, operand
typing, generator indices, rational literal canonicality and multiple
output roots. The evaluator handles integer registers and all six
decoded instruction types. For every accepted root, the companion
proves successful evaluation and equivalence between equality to a
supplied target and a zero test on the corresponding skew expression.
A further module checks the exact propagated maximum of the formal
path exponents through the decoded program. The maximum passes a
strict comparison exactly when every formal path passes it, including
paths retained through a zero scalar. The generator leading exponents
and the pair limit are supplied to this check.

The certificate-length components include a complete proof of
Lemma~\ref{lem:dense-bit-cost}, with the explicit constant \(4096\).
The companion also checks the subset-sum path bound for resolved
cofactor register programs and the exact doubling-label cost
\(2C=k(k+1)\). These checks supply components of the full encoded
certificate-length arguments.
For the bivariate family, the companion proves uniqueness of the
strict cofactor tuple for every degree parameter. It constructs
actual cofactor pairs on resolved register programs and proves
the subset-sum invariant for their surviving support. This yields
the quadratic lower bound for the resolved shift-label cost and
the linear lower bound for the number of shifts. When every
resolved shift label is stored explicitly as a gamma field, the
same quadratic lower bound holds for those field bits. The
formal statement retains that storage hypothesis; shared exponent
instructions require their separate length analysis.

The companion checks the shared-exponent construction as a binary
instruction stream. Its instruction portion has exactly
\(14+56k\) bits. A complete counted frame with one output root
has at most \(33+76k\) bits and is accepted by the validated
decoder. The companion proves the required polynomial identity
and strict bound for its root with both basis leaves available.

For arbitrary accepted pair frames, augmented cofactor records
connect the decoded evaluation to the forced-cofactor argument.
Their support, leading-bound and path-count invariants yield
\(3k\le L_k\le109k\) for \(k\ge1\), as stated in
Proposition~\ref{prop:explicit-frame-length}. This proves the
matching linear length claim for pair witnesses in that concrete
format. The additional containment roots of a complete certificate
remain a separate formalization task.

The formalization currently covers selected components. The
end-to-end bit-time bounds, the randomized complexity transfer and
the remaining optimal certificate-length claims remain supported by
the written arguments. The supplementary \texttt{COVERAGE.md} maps
the checked statements and the remaining gaps. The complexity
equivalence in Theorem~\ref{thm:verification-pit} has not been
formalized end to end.